\documentclass[12pt,amstex]{amsart}

\usepackage{amsmath, amsthm, amssymb, mathrsfs, enumerate, bm, enumitem}
\usepackage{comment}
\usepackage{color}
\usepackage[dvipdfmx]{graphicx}
\usepackage{xparse}
\usepackage{mathtools}

\usepackage{a4wide}
\newtheorem{theorem}{Theorem}[section]
\newtheorem{corollary}[theorem]{Corollary}

\newtheorem{proposition}[theorem]{Proposition}

\theoremstyle{definition}

\theoremstyle{remark}

\DeclareMathOperator{\Z}{\mathbb{Z}}

\DeclareMathOperator{\C}{\mathbb{C}}

\DeclareMathOperator{\m}{\mathrm{m}}
\DeclareMathOperator{\degg}{\mathrm{deg}}

\DeclareMathOperator{\D}{\mathrm{D}}

\DeclarePairedDelimiterX\set[1]\lbrace\rbrace{\def\given{\;\delimsize\vert\;}#1}

\begin{document}
\title[Ihara's formula for simple graphs with bounded degree]
{A generalized Ihara zeta function formula for simple graphs with bounded degree}

\author[T.~Kousaka]
{Taichi Kousaka}

\address{Graduate School of Mathematics, Kyushu University,
744 Motooka, Nishi-ku, Fukuoka, 819-0395 JAPAN}
\email{t-kosaka@math.kyushu-u.ac.jp}


\begin{abstract}
We establish a generalized Ihara zeta function formula for simple graphs with bounded degree. 
This is a generalization of the formula obtained by G.~Chinta, J.~Jorgenson and A.~Karlsson from vertex-transitive graphs. 
\end{abstract}

\maketitle
\section{Introduction}
Let $X$ be a connected graph with bounded degree whose vertex set is countable and $\Delta_{X}$ be the combinatorial Laplacian on $X$. 
In this paper, a graph with bounded degree means a graph which has the above properties. 
The relationship between geometric properties of $X$ and the spectrum of $\Delta_{X}$ has been widely studied. 
Especially, for a finite regular graph, it is well-known that the distribution of the spectrum of $\Delta_{X}$ is deeply related to the number of closed geodesics (cf.~\cite{Terras}). 
The Ihara zeta function for a finite graph is defined by
\begin{align*}
Z_{X}(u)=\exp \bigg(\sum_{m=1}^{\infty}\frac{N_{m}}{m}u^{m} \bigg). 
\end{align*}
Here, we denote by $N_{m}$ the number of closed geodesics of length $m$ in $X$. 
This function is directly related to the number of closed geodesics. 
The original Ihara zeta function was first defined by Y.~Ihara in \cite{YI1966} as a Selberg-type zeta function in the $p$-adic setting. 
It can be interpreted in terms of finite regular graphs and has been generalized by T.~Sunada, K.~Hashimoto and H.~Bass (\cite{Bass1992}, \cite{HH89}, \cite{H89}, \cite{H90}, \cite{H92}, \cite{H93}, \cite{KS2000}, \cite{serre1}, \cite{sunada86}, \cite{sunada88}). 
There are various studies for the Ihara zeta function for a finite graph. 
The most famous and important formula for the Ihara zeta function for a finite graph is the Ihara determinant formula described as 
\begin{align*}
Z_{X}(u)^{-1}=(1-u^2)^{-\chi(X)}\det\big(I-u(D_{X}-\Delta_{X})+u^{2}(D_{X}-I)\big). 
\end{align*}
Here, we denote by $\chi(X)$ the Euler number of $X$ and by $\D_{X}$ the valency operator on $X$. 
The above formula for a finite regular graph was originally established by Y.~Ihara in the $p$-adic setting (\cite{YI1966}). 
Various proofs of this formula are well-known (cf.~\cite{Bass1992}, \cite{KS2000}). 
This formula can be interpreted as a formula describing a relationship between the number of closed geodesics and the spectrum of the Laplacian on a finite graph. 

\medskip
Recently, several authors have considered generalizations of the Ihara zeta function and the Ihara determinant formula from finite graphs to infinite graphs (cf.~\cite{CJK1}, \cite{CMS2001}, \cite{Cla2009}, \cite{Deit2015}, \cite{GZ2004}, \cite{GIL2008a}, \cite{GIL2008b}, \cite{Sch1999}). 
Among them, we follow \cite{CJK1} essentially. 
In \cite{CJK1}, for a regular graph, the Ihara zeta function is defined. 
The definition is complicated because we have to introduce another terminology besides closed geodesics (cf.~p.~$185$ in \cite{CJK1}). 
However, if a graph $X$ is a vertex-transitive graph, the Ihara zeta function defined in \cite{CJK1} has a natural form. 
Namely, for a vertex-transitive graph $X$, the Ihara zeta function of $X$ defined in \cite{CJK1} is  
\begin{align*}
\zeta_{X}(u)=\exp\bigg(\sum_{m=1}^{\infty} \frac{N_{m}(x_{0})}{m}u^{m}\bigg). 
\end{align*}
Here, we denote by $N_{m}(x_{0})$ the number of closed geodesics of length $m$ starting at a given vertex $x_{0}$. 
We remark that the above zeta function does not depend on the given vertex since $X$ is a vertex-transitive graph. 
We also remark that if a regular graph $X$ is not a vertex-transitive graph, the zeta function of $X$ defined in \cite{CJK1} does not always coincide with the above. 
The idea of giving a generalization of the Ihara zeta function from finite graphs to infinite graphs in \cite{CJK1} is to count not all closed geodesics but only count through a fixed starting vertex. 
The Ihara's formula for the zeta function defined in \cite{CJK1} is proved by giving a new expression of the heat kernel on a regular graph by using modified Bessel functions. 
The approach through heat kernel analysis is considered to be successful also from the view point of giving an analogy with quotients of rank one symmetric spaces. 
Therefore, we define the zeta function for a graph with bounded degree following \cite{CJK1}. 

\medskip
The aim of this paper is to continue the study about the relationship between the number of closed geodesics and the spectrum of the Laplacian on a graph. 
From this standpoint, for a graph $X$ with bounded degree and a vertex $x_{0} \in V\!X$, we define the Ihara zeta function as follows in this paper. 
\begin{align*}
Z_{X}(u, x_{0})=\exp\bigg(\sum_{m=1}^{\infty} \frac{N_{m}(x_{0})}{m}u^{m}\bigg). 
\end{align*}
We remark that the above zeta function depends on the vertex $x_{0}$. 
As mentioned in \cite{CJK1}, this generalization of the Ihara zeta function can be considered as corresponding to the Hurwitz zeta function which is generalized from the Riemann zeta function. 
Moreover, as we mentioned above, the Ihara zeta function for a graph $X$ with bounded degree is equal to the Ihara zeta function defined in \cite{CJK1} if $X$ is a vertex-transitive graph. 
However, our zeta function does not always coincide with the zeta function defined in \cite{CJK1} for a regular graph which is not a vertex-transitive graph. 

In this paper, we establish a generalized Ihara zeta function formula for connected simple graphs with bounded degree by using an algebraic method. 
Here, a simple graph means a graph which have no loops and multiple edges. 
This is a generalization of the formula for vertex-transitive graphs obtained by G.~Chinta, J.~Jorgenson and A.~Kerlsson in \cite{CJK1}. 
We remark that we establish the formula for connected simple graphs with bounded degree whereas G.~Chinta, J.~Jorgenson and A.~Karlsson establish the formula for connected vertex-transitive graphs which are not always simple graphs (\cite{CJK1}). 
Moreover, our proof also gives an alternative proof of the formula obtained by G.~Chinta, J.~Jorgenson and A.~Karlsson for simple vertex-transitive graphs. 
For finite simple regular graphs, we also derive a generalized Ihara zeta function formula which is regarded as a local version of the original Ihara determinant formula.  

\section{Preliminaries}
\subsection{Graphs and Paths}
In this section, we give terminology of graphs and paths used throughout in this paper (cf.~\cite{serre1}, \cite{sunada1}). 
A graph $X$ is an ordered pair $(V\!X, E\!X)$ of disjoint sets $V\!X$ and $E\!X$ with two maps, 
\begin{align*}
E\!X \rightarrow V\!X \times V\!X, e \mapsto (o(e), t(e)), \quad
E\!X \rightarrow E\!X , e \mapsto \bar{e}
\end{align*}
such that for each $e \in E\!X$, $\bar{e} \neq e$, $\bar{\bar{e}} = e$, $o(e)=t(\bar{e})$. 
For a graph $X=(V\!X, E\!X)$, two sets $V\!X$ and $E\!X$ are called vertex set and edge set respectively. 
For a vertex $x \in V\!X$, the {\it degree of $x$} is the cardinality of the set $E_{x}$, where $E_{x}=\{ e \in E\!X | o(e)=x\}$. We denote the degree of $x$ by $\degg(x)$. 
A graph $X$ is {\it countable} if the vertex set is countable. 
A graph $X$ {\it has bounded degree} if the supremum of the set of all degrees is not infinite. 
All graphs which are considered in this paper will be connected, countable and have bounded degree without degree one vertices. 
As we stated in Section $1$, a graph with bounded degree means a graph which has the above properties in this paper. 
For a graph $X$, a {\it path of length $n$} is a sequence of edges
\begin{align*}
c=(e_{1}, \dots, e_{n})
\end{align*}
such that $t(e_{i})=o(e_{i+1})$ for each $i$. 
We denote $o(e_{1})$ by $o(c)$, $t(e_{n})$ by $t(c)$ and the length of $c$ by $\ell(c)$. 
A path $c$ is {\it closed} if $o(c)=t(c)$. 
We regard a vertex as a path of length $0$. 
A path $c=(e_{1}, \dots, e_{n})$ {\it has a back-tracking} if there exist $i$ such that $e_{i+1}=\bar{e_{i}}$. A path $c=(e_{1}, \dots, e_{n})$ {\it has a tail} if $e_{n}=\bar{e_{1}}$. 
A path $c$ is a {\it geodesic} if $c$ has no back-tracking. 
A closed path $c=(e_{1}, \dots, e_{n})$ is a {\it geodesic loop} if $c$ has no back-tracking. 
A closed path $c=(e_{1}, \dots, e_{n})$ is a {\it closed geodesic path} if $c$ is a geodesic loop and has no tail. 

\subsection{The combinatorial Laplacian on a graph}
For the vertex set $V\!X$ of a graph $X$, we define {\it the $\ell^2$-space on the vertex set $V\!X$} by 
\begin{align*}
\ell^{2}(V\!X)=\set[\bigg]{ f \colon V\!X \rightarrow \C \given \sum_{x \in V\!X}|f(x)|^2 < +\infty }. 
\end{align*}
For a function $f \in \ell^2(V\!X)$ and a vertex $x \in V\!X$, we define the {\it adjacency operator} $A_{X}$ on $X$ and the {\it valency operator} $D_{X}$ on $X$ as follows respectively. 
\begin{align*}
(A_{X}f)(x)=\sum_{e \in E_{x}}f(t(e)), \\
(D_{X}f)(x)=\degg(x)f(x). 
\end{align*}
Then,  we define the {\it combinatorial Laplacian} $D_{X}$ on $X$ by $\Delta_{X}=D_{X}-A_{X}$. 
The Laplacian is a semipositive and self-adjoint bounded operator under our assumption. 

\section{a path counting formula}
In this section, we give an explicit formula between the number of geodesic loops and the number of closed geodesics. 
We remark that a graph which is considered in this section is allowed to have multiple edges and loops. 
This is a generalization of the path counting formula obtained in \cite{CJK1}. 
We fix a vertex $x_{0} \in V\!X$. 
For a vertex $x \in V\!X$ and a nonnegative integer $m$, 
we denote by $c_{m}(x_{0}, x)$ the number of geodesic paths of length $m$ from $x_{0}$ to $x$, by $c_{m}(x)=c_{m}(x, x)$ the number of geodesic loops of length $m$ starting at $x$ and by $N_{m}(x_{0})$ the number of closed geodesics of length $m$ starting at $x_{0}$. 
For a vertex $x \in V\!X$ and a nonnegative integer $m$, 
we denote $\degg(x)c_{m}(x)-\sum_{e \in E_{x}}c_{m}(t(e))$(resp.~$\degg(x)N_{m}(x)-\sum_{e \in E_{x}}N_{m}(t(e))$) by $(\Delta_{X}c_{m})(x)$(resp.~$(\Delta_{X}N_{m})(x))$ formally 
by regarding $c_{m}$ and $N_{m}$ as functions on $V\!X$. 
Here, we note that functions $c_{m}$ and $N_{m}$ are not in $\ell^{2}(V\!X)$ in general. 
Moreover, we define the following formal power series. 
\begin{align*}
C(u: x_{0})=\sum_{m=1}^{\infty}c_{m}(x_{0})u^m, \\
N(u: x_{0})=\sum_{m=1}^{\infty}N_{m}(x_{0})u^m. 
\end{align*}
Our goal in this section is to give the following theorem. 
\begin{theorem}\label{NformulaM}
The following formula holds: 
\begin{align*}
N(u: x_{0})=(1-u^2)^{-2}\{ 1-(\degg(x_{0})-\Delta_{X})u^2+(\degg(x_{0})-1)u^4\}C(u: \cdot)(x_{0}). 
\end{align*}
\end{theorem}

First of all, we give the following proposition. 
\begin{proposition}\label{Nformula}
For an integer $m$ greater than $2$, the following identity holds: 
\begin{align*}
N_{m}(x_{0})=c_{m}(x_{0})-(\degg(x_{0})-2)\sum_{i=1}^{\lceil \frac{m}{2} \rceil-1}c_{m-2i}(x_{0})
			+\sum_{i=1}^{\lceil \frac{m}{2} \rceil-1}i(\Delta_{X}c_{m-2i})(x_{0}). 
\end{align*}
Here, the symbol $\lceil \cdot \rceil$ stands for the ceiling function. 
\end{proposition}

\begin{proof}
First of all, we define several symbols. For a non-negative integer $m$, a vertex $x \in V\!X$ and an edge $e \in E\!X$ such that $o(e)=x$ or $t(e)=x$, 
we denote by $c_{m}(x, e)$(resp.~$N_{m}(x, e)$) the number of geodesic loops (resp.~closed geodesics) of length $m$ starting at $x$ through $e$. 
Let $m$ be an integer which is greater than $2$ and $e \in E\!X$ be an edge such that $o(e)=x_{0}$. 
The number $c_{m}(x_{0}, e)-N_{m}(x_{0}, e)$ is equal to the number of geodesic loops of length $m$ starting at $x_{0}$ through $e$, which are not closed geodesics. 
Therefore, we have 
\begin{align*}
c_{m}(x_{0}, e)-N_{m}(x_{0}, e)&=c_{m-2}(t(e))-c_{m-2}(t(e), \bar{e})-c_{m-2}(t(e), e)
	+cwt_{m-2}(t(e), \bar{e}).  
\end{align*}
Here, we denote by $cwt_{m-2}(t(e), \bar{e})$ 
the number of geodesic loops of length $m-2$ starting at $t(e)$ with tail $\bar{e}$. 
By this, we have 
\begin{align*}
	&c_{m}(x_{0})-N_{m}(x_{0})\\
	&=\sum_{e \in E_{x_{0}}}\{ c_{m-2}(t(e))-c_{m-2}(t(e), \bar{e})-c_{m-2}								(t(e), e)+cwt_{m-2}(t(e), \bar{e}) \} \\
	&=\degg(x_{0})c_{m-2}(x_{0})-(\Delta_{X}c_{m-2})(x_{0})\\
	   &\quad -\sum_{e \in E_{x_{0}}}\{c_{m-2}(t(e), \bar{e})-N_{m-2}(t(e), \bar{e})\}
	      -\sum_{e \in E_{x_{0}}}\{c_{m-2}(t(e), e)-N_{m-2}(t(e), e)\} \\
	    &\quad-\sum_{e \in E_{x_{0}}}\{N_{m-2}(t(e), \bar{e})+N_{m-2}(t(e), e)\}
								+\sum_{e \in E_{x_{0}}}cwt_{m-2}(t(e), \bar{e}) . 
\end{align*}
Further, 
\begin{align*}
	&\sum_{e \in E_{x_{0}}}\{c_{m-2}(t(e), \bar{e})-N_{m-2}(t(e), \bar{e})\}\\
	&=\sum_{e \in E_{x_{0}}}\{c_{m-2}(t(e), e)-N_{m-2}(t(e), e)\}\\
	&= cwt_{m-2}(t(e), \bar{e}) 
\end{align*}
and 
\begin{align*}
	\sum_{e \in E_{x_{0}}}N_{m-2}(t(e), \bar{e})
	=\sum_{e \in E_{x_{0}}}N_{m-2}(t(e), e)
	=\sum_{e \in E_{x_{0}}}N_{m-2}(x_{0}, e)=N_{m-2}(x_{0}). 
\end{align*}
Then, we have 
\begin{align}\label{ceq1}
&c_{m}(x_{0})-N_{m}(x_{0})\nonumber \\
&=\degg(x_{0})c_{m-2}(x_{0})-(\Delta_{X}c_{m-2})(x_{0})-2N_{m-2}(x_{0})
					-\sum_{e \in E_{x_{0}}}cwt_{m-2}(t(e), \bar{e}). 
\end{align}
Putting $m=3$ in $(\ref{ceq1})$, since $cwt_{1}(t(e), \bar{e})=0$, we have 
\begin{align*}
c_{3}(x_{0})-N_{3}(x_{0})=\degg(x_{0})c_{1}(x_{0})-(\Delta_{X}c_{1})(x_{0})-2N_{1}(x_{0}). 
\end{align*}
Therefore, 
\begin{align}\label{N3}
N_{3}(x_{0})=c_{3}(x_{0})-(\degg(x)-2)c_{1}(x_{0})+(\Delta_{X}c_{1})(x_{0}). 
\end{align}
By the same argument, we have
\begin{align*}
N_{4}(x_{0})=c_{4}(x_{0})-(\degg(x_{0})-2)c_{2}(x_{0})+(\Delta_{X}c_{2})(x_{0}). 
\end{align*}
In the case $m \geq 5$, by $(\ref{ceq1})$, we have 
\begin{align*}
&c_{m}(x_{0})-N_{m}(x_{0})\\
&=\degg(x_{0})c_{m-2}(x_{0})-(\Delta_{X}c_{m-2})(x_{0})-2N_{m-2}(x_{0})\\
&\quad-\{N_{m-4}(x_{0})\degg(x_{0})-2)+(c_{m-4}-N_{m-4}(x_{0}))(\degg(x_{0})-1)\} \\
&=\degg(x_{0})c_{m-2}(x_{0})-(\Delta_{X}c_{m-2})(x_{0})\\
&\quad-2N_{m-2}(x_{0})-(\degg(x_{0})-1)c_{m-4}(x_{0})+N_{m-4}(x_{0}). 
\end{align*}
Therefore, we have the following recursive formula. 
\begin{align*}
&\{N_{m}(x_{0})-N_{m-2}(x_{0})\}-\{N_{m-2}(x_{0})-N_{m-4}(x_{0})\} \\
&=\{c_{m}(x_{0})-c_{m-2}(x_{0})\}-(\degg(x_{0})-1)\{c_{m-2}(x_{0})-c_{m-4}(x_{0})\}
+(\Delta_{X}c_{m-2})(x_{0}). 
\end{align*}
For $m \geq 5$ which is an odd integer, we have 
\begin{align*}
&N_{m}(x_{0})-N_{m-2}(x_{0})\\
&=c_{m}(x_{0})-(\degg(x_{0})-1)c_{m-2}(x_{0})+(N_{3}(x_{0})-c_{3}(x_{0}))\\
&\quad+(\degg(x_{0})-1)c_{1}(x_{0})-N_{1}(x_{0})+\sum_{i=1}^{\frac{m-3}{2}}(\Delta_{X}c_{m-2i})(x_{0}). 
\end{align*}
Here, we used the above recursive formula in the above equation. 
By this and $(\ref{N3})$, we have
\begin{align*}
N_{m}(x_{0})=c_{m}(x_{0})-(\degg(x_{0})-2)\sum_{i=1}^{\frac{m-1}{2}}c_{m-2i}(x_{0})
			+\sum_{i=1}^{\frac{m-1}{2}}i(\Delta_{X}c_{m-2i})(x_{0}). 
\end{align*}
For $m \geq 6$ which is an even integer, by the same argument, we have
\begin{align*}
N_{m}(x_{0})=c_{m}(x_{0})-(\degg(x_{0})-2)\sum_{i=1}^{\frac{m-2}{2}}c_{m-2i}(x_{0})
			+\sum_{i=1}^{\frac{m-2}{2}}i(\Delta_{X}c_{m-2i})(x_{0}). 
\end{align*}
\end{proof}

In the rest of this section, we give the proof of Theorem \ref{NformulaM}. 
We define $R_{m}(x_{0})$ and $\tilde{R}_{m}(x_{0})$ as follows. 
\begin{align*}
R_{m}(x_{0})=\sum_{i=1}^{\lceil \frac{m}{2} \rceil-1}i(\Delta_{X}c_{m-2i})(x_{0}), \\
\tilde{R}_{m}(x_{0})=\begin{cases}
				R_{1}(x_{0})	& \text{if $m=1$, } \\
				R_{2}(x_{0})	& \text{if $m=2$, } \\
				R_{m}(x_{0})-R_{m-2}(x_{0})	& \text{if $m \geq 3$. }
				\end{cases}
\end{align*}
We define the corresponding formal power series as follows. 
\begin{align*}
R(u: x_{0})=\sum_{m=1}^{\infty}R_{m}(x_{0})u^m, \\
\tilde{R}(u: x_{0})=\sum_{m=1}^{\infty}\tilde{R}_{m}(x_{0})u^m. 
\end{align*}
By the definition of $R_{m}(x_{0})$ and $R_{m}(x_{0})$, we have 
\begin{align}\label{Rformula}
R(u: x_{0})=u^2(1-u^2)^{-2}\Delta_{X}C(u: x_{0}). 
\end{align}
Moreover, for vertices $x_{0}, x \in V\!X$, we define $b_{m}(x_{0})$ as follows. 
\begin{align*}
b_{m}(x)=\begin{cases}
			c_{0}(x_{0}, x)	& \text{if $m=0$, } \\
			c_{1}(x_{0}, x)	& \text{if $m=1$, } \\
			c_{2}(x_{0}, x)	& \text{if $m=2$, } \\
			c_{m}(x_{0}, x)-(\degg(x)-2)\sum_{j=1}^{\lceil \frac{m}{2} \rceil-1}c_{m-2j}(x_{0}, x)
						& \text{if $m \geq 3$. } \\
		\end{cases}
\end{align*}
By the definition of this symbol, we have 
\begin{align}\label{Bformula}
B(u: x_{0})=(1-u^2)^{-1}\{1-(\degg(x_{0})-1)u^2\}C(u: x_{0}). 
\end{align}
By Proposition \ref{Nformula}, we have 
\begin{align*}
N(u: x_{0})=C(u: x_{0})-(C(u: x_{0})-B(u: x_{0}))+R(u: x_{0})=B(u: x_{0})+R(u: x_{0}). 
\end{align*}
By (\ref{Rformula}) and (\ref{Bformula}), we have 
\begin{align*}
N(u: x_{0})&=(1-u^2)^{-1}\{1-(\degg(x_{0})-1)u^2\}C(u: x_{0})
			+u^{2}(1-u^{2})^{-2}\Delta_{X}C(u: \cdot )(x_{0}) \\
		&=(1-u^2)^{-2}\{1-(\degg(x_{0})-\Delta_{X})u^2
					+(\degg(x_{0})-1)u^4\}C(u: \cdot)(x_{0}). 
\end{align*}
Therefore, we get Theorem \ref{NformulaM}. 

\section{A generalized Ihara zeta function formula for simple graphs with bounded degree}
In this section, we give a generalized Ihara zeta function formula for a simple graph with bounded degree. 
Let $X$ be a connected simple graph with bounded degree. 
We denote the supremum of all degrees of $X$ by $M$. 
We remark that $M$ is greater than $1$ by our assumption. 
We denote the set of all bounded operators on $\ell^{2}(V\!X)$ equipped with the usual operator norm $\lVert \cdot \rVert$ by $ \mathcal{B}(\ell^{2}(V\!X))$. 

First of all, we introduce several bounded operators on $\ell^{2}(V\!X)$. 
For $f \in \ell^{2}(V\!X)$ and $m \in \Z_{\geq 0}$, we define $C_{m}$ as follows. 
\begin{align*}
C_{m}f(x)=\sum_{c \in \mathcal{C}_{x}, \ell(c)=m} f(t(c)). 
\end{align*}
Here, the symbol $\mathcal{C}_{x}$ stands for the set of all geodesic paths of length $m$ starting at $x$. 
We define $Q_{X}$ by $D_{X}-I$ and $B_{m}$ as follows. 
\begin{align*}
B_{m}=\begin{cases}
		C_{m}-(Q-I)\sum_{j=1}^{\lfloor \frac{m}{2} \rfloor}C_{m-2j} & \text{if $m \geq 3$, } \\
		C_{m}	& \text{if $m=0, 1, 2$. }
	     \end{cases} 
\end{align*}
For $f \in \ell^{2}(V\!X)$, we define $R_{m}$ as follows. 
\begin{align*}
(R_{m}f)(x)=\begin{cases}
		\sum_{j=1}^{\lceil \frac{m}{2} \rceil-1}j(\Delta_{X}c_{m-2j})(x)f(x)	 & \text{if $m \geq 3$, } \\
		0	& \text{if $m=0, 1, 2$. }
	     \end{cases}
\end{align*}
Moreover, we define $R_{m}^{+}$ and $N_{X, m}$ as follows. 
\begin{align*}
R_{m}^{+}=\begin{cases}
		(Q_{X}-I)\delta_{2\Z}(m)+R_{m} & \text{if $m \geq 3$, } \\
		0	& \text{if $m=0, 1, 2$}, 
	     \end{cases}
\end{align*}
\begin{align*}
N_{X, m}=B_{m}+R_{m}^{+}. 
\end{align*}
We remark that the above operators are in $\mathcal{B}(\ell^{2}(V\!X))$ since $X$ has a bounded degree. 
For $B \in \mathcal{B}(\ell^{2}(V\!X))$ and for $x_{0}, x \in V\!X$, we define $B(x_{0}, x)$ as follows. 
\begin{align*}
B(x_{0}, x)=B\delta_{x_{0}}(x). 
\end{align*}
Here, the symbol $\delta_{x_{0}}$ stands for the Kronecker delta. 
We remark that $B(x_{0}, x)$ is in $\C$ by the Cauchy-Schwarz inequality since $B$ is in $\mathcal{B}(\ell^{2}(V\!X))$. 

Then, we have the following proposition. 
\begin{proposition}$($\cite{GIL2008a}$)$\label{C_{m}}
We have the following equation: 
\begin{align*}
C_{m}=\begin{cases}
		C_{1}^{2}-Q-I	& \text{if $m=2$, } \\
		C_{m-1}C_{1}-C_{m-2}Q	& \text{if $m \geq 3$. }
	     \end{cases}
\end{align*}
Let $\alpha=\frac{M+\sqrt{M^{2}+4M}}{2}$. Then, for $m \in \Z_{\geq 0}$, we have 
\begin{align*}
\lVert C_{m} \rVert \leq \alpha^{m}. 
\end{align*}
Moreover, for $\left| u \right| < \frac{1}{\alpha}$, we have the following equations: 
\begin{enumerate}[label={\rm (\arabic*)}]
\item $\bigg(\sum_{m=0}^{\infty}C_{m}u^m\bigg)\big(I-uA_{X}+u^{2}Q_{X}\big)=(1-u^{2})I$. 
\item $\bigg(\sum_{m=0}^{\infty}\big(\sum_{k=0}^{\lfloor \frac{m}{2} \rfloor}C_{m-2k}\big)u^{m}\bigg)
	\big(I-uA_{X}+u^{2}Q_{X}\big)=I$. 
\end{enumerate}
\end{proposition}

By Proposition \ref{C_{m}}, we have the following proposition. 

\begin{proposition}\label{N_{m}}
\begin{enumerate}[label={\rm (\arabic*)}]
\item For $\left| u \right| < \frac{1}{\alpha}$, we have 
	\begin{align*}
	\bigg(\sum_{m=1}^{\infty}B_{m}u^m\bigg)\big(I-uA_{X}+u^{2}Q_{X}\big)
	=A_{X}u-2Q_{X}u^2+\big(Q-I\big)\big(I-uA_{X}+u^{2}Q_{X}\big)u^2. 
	\end{align*} 
\item For $\left| u \right| < \frac{1}{\alpha}$, we have 
	\begin{align*}
	\sum_{m=1}^{\infty}N_{X, m}u^m=u(A_{X}-2Q_{X}u)\big(I-uA_{X}+u^{2}Q_{X}\big)^{-1}
			+(Q_{X}-I)\frac{u^2}{1-u^2}+\sum_{m=3}^{\infty}R_{m}u^m.
		\end{align*}
\end{enumerate}

\end{proposition}
It is easy to check by Proposition \ref{C_{m}}. Therefore, we omit the proof of Proposition \ref{N_{m}} 

Let $f$ be a $C^{1}$-function on $B_{\epsilon}=\{ u \in \C \big| \left| u \right| < \epsilon  \}$ which takes the value to bounded operators on a Hilbert space and satisfies $f(0)=0$, $\lVert f(u) \rVert < 1$ for any $u \in B_{\epsilon}$. 
Here, $\lVert \cdot \rVert$ stands for the operator norm on this Hilbert space. 
Then, for $u \in B_{\epsilon}$, we have 
\begin{align*}
-\log(I-f(u))=\sum_{n=1}^{\infty}\frac{1}{n}f(u)^{n}. 
\end{align*}
Here, the above series converges in operator norm, uniformly on compact subsets of $B_{\epsilon}$. 
By this, we have 
\begin{align*}
-\frac{d}{du}\log(I-f(u))=\sum_{n=1}^{\infty}\frac{1}{n}\sum_{j=0}^{n-1}f(u)^{j}f^{'}(u)f(u)^{n-j-1}. 
\end{align*}

Let $f(u)=A_{X}u-Q_{X}u^2$. We remark that $\left| u \right| < \frac{1}{\alpha}$ implies $\lVert f(u) \rVert < 1$. 
Then, we have the following proposition. 
\begin{proposition}\label{f(u)}
For $\left| u \right| < \frac{1}{\alpha}$, we have 
\begin{align*}
f^{'}(u)(I-f(u))^{-1}=-\frac{d}{du}\log(I-f(u))
	+u^2\sum_{n=1}^{\infty}\frac{1}{n}\sum_{j=1}^{n-1}jf(u)^{n-1-j}[A_{X}, Q_{X}]f(u)^{j-1}. 
\end{align*}
Here, $[A_{X}, Q_{X}]=A_{X}Q_{X}-Q_{X}A_{X}$. 
\end{proposition}
\begin{proof}
By the previous remark, for $\left| u \right| < \frac{1}{\alpha}$, we have 
\begin{align*}
-\frac{d}{du}\log(I-f(u))=\sum_{n=1}^{\infty}\frac{1}{n}\sum_{j=0}^{n-1}f(u)^{j}f^{'}(u)f(u)^{n-j-1}. 
\end{align*}
By straightforward calculation, we have 
\begin{align*}
[f(u), f^{'}(u)]=(Q_{X}A_{X}-A_{X}Q_{X})u^2. 
\end{align*}
Therefore, we get 
\begin{align*}
f(u)f^{'}(u)=f^{'}(u)f(u)+[Q_{X}, A_{X}]u^2. 
\end{align*}
By this equation, we have 
\begin{align*}
\sum_{j=0}^{n-1}f(u)^{j}f^{'}(u)f(u)^{n-1-j}=nf^{'}(u)f(u)^{n-1}
	+u^{2}\sum_{j=1}^{n-1}jf(u)^{n-1-j}[Q_{X}, A_{X}]f(u)^{j-1}. 
\end{align*}
Then, we have 
\begin{align*}
-\frac{d}{du}\log(I-f(u))&=\sum_{n=1}^{\infty}\frac{1}{n}\big( nf^{'}(u)f(u)^{n-1}
	+u^{2}\sum_{j=1}^{n-1}jf(u)^{n-1-j}[Q_{X}, A_{X}]f(u)^{j-1}\big)\\
				&=f^{'}(u)\sum_{n=1}^{\infty}f(u)^{n-1}
				+u^2\sum_{n=1}^{\infty}\frac{1}{n}\sum_{j=1}^{n-1}jf(u)^{n-1-j}[Q_{X}, A_{X}]f(u)^{j-1}\\
				&=f^{'}(u)(I-f(u))^{-1}
				 +u^2\sum_{n=1}^{\infty}\frac{1}{n}\sum_{j=1}^{n-1}jf(u)^{n-1-j}[Q_{X}, A_{X}]f(u)^{j-1}. 
\end{align*}
\end{proof}
By Proposition \ref{N_{m}} and Proposition \ref{f(u)}, for $\left| u \right| < \frac{1}{\alpha}$, we have 
\begin{align*}
u\frac{d}{du}\sum_{m=1}^{\infty}\frac{N_{X, m}}{m}u^{m}
&=-u\frac{d}{du}\log(I-f(u))
	+u^3\sum_{n=1}^{\infty}\frac{1}{n}\sum_{j=1}^{n-1}jf(u)^{n-1-j}[A_{X}, Q_{X}]f(u)^{j-1}\\
	&\quad+(Q_{X}-I)\frac{u^2}{1-u^2}+u\frac{d}{du}\sum_{m=3}^{\infty}\frac{R_{m}}{m}u^m. 
\end{align*}
Dividing by $u$ and integrating from $u=0$ to $u$, we have
\begin{align*}
\sum_{m=1}^{\infty}\frac{N_{X, m}}{m}u^m
&=-\log(I-f(u))-\frac{Q_{X}-I}{2}\log(1-u^2)\\
&\quad+\int_{0}^{u}z^2\sum_{n=1}^{\infty}\frac{1}{n}\sum_{j=1}^{n-1}jf(z)^{n-1-j}[A_{X}, Q_{X}]f(z)^{j-1}dz
+\sum_{m=3}^{\infty}\frac{R_{m}}{m}u^m. 
\end{align*}
Therefore, for $x_{0}, x \in V\!X$, we have 
\begin{align}\label{N_{x_{0}, x}} \nonumber
&\sum_{m=1}^{\infty}\frac{N_{X, m}(x_{0}, x)}{m}u^m\\ \nonumber
&=-[\log(I-A_{X}u+Q_{X}u^2)](x_{0}, x)-\frac{\degg(x_{0})-2}{2}\delta{x_{0}}(x)\log(1-u^2)\\
&\quad+\int_{0}^{u}z^2\sum_{n=1}^{\infty}
	\frac{1}{n}\sum_{j=1}^{n-1}j\big[f(z)^{n-1-j}[A_{X}, Q_{X}]f(z)^{j-1}\big](x_{0}, x)dz
	+\sum_{m=3}^{\infty}\frac{R_{m}(x_{0}, x)}{m}u^m. 
\end{align}
We define $Z_{X}(u, x_{0}, x)$ as follows. 
\begin{align*}
Z_{X}(u, x_{0}, x)=\exp\bigg(\sum_{m=1}^{\infty}\frac{N_{X, m}(x_{0}, x)}{m}u^m \bigg)
\end{align*}
We remark that $Z_{X}(u, x_{0}, x_{0})=Z_{X}(u, x_{0})$ by Proposition \ref{Nformula}. 
Then, we have the following theorem by (\ref{N_{x_{0}, x}}). 
\begin{theorem}\label{Ihara formula}
For $\left| u \right| < \frac{1}{\alpha}$, we have 
\begin{align*}
Z_{X}(u, x_{0}, x)&=(1-u^2)^{-\frac{\degg(x_{0})-2}{2}\delta_{x_{0}}(x)}\\
	&\quad \times \exp\big(-[\log(I-(D_{X}-\Delta_{X})u+(D_{X}-I)u^2)](x_{0}, x)\big) \\
		&\quad \times \exp\bigg( \int_{0}^{u}z^2\sum_{n=1}^{\infty}
	\frac{1}{n}\sum_{j=1}^{n-1}j\big[f(z)^{n-1-j}[A_{X}, D_{X}]f(z)^{j-1}\big](x_{0}, x)dz\bigg)\\
	&\quad \times \exp\bigg(\sum_{m=3}^{\infty}\frac{R_{m}(x_{0}, x)}{m}u^m\bigg). 
\end{align*}
\end{theorem}

In particular, if $X$ is a $(q+1)$-regular graph, we have 
\begin{align*}
I-(D_{X}-\Delta_{X})u+Q_{X}u^2=I-\big((q+1)I-\Delta_{X}\big)u+qIu^2. 
\end{align*}
Since $\Delta_{X}$ is a self-adjoint bounded operator, there exists a unique spectral measure such that 
\begin{align*}
\Delta_{X}=\int_{\sigma(\Delta_{X})}\lambda dE(\lambda). 
\end{align*}
Here, $\sigma(\Delta_{X})$ stands for the spectram of $\Delta_{X}$. 
For $x_{0}, x \in V\!X$, we denote the measure $\langle E(\cdot)\delta_{x_{0}}, \delta_{x} \rangle$ by $\mu_{x_{0}, x}(\cdot)$. 
Then, by the property of the spectral integral, we have the following corollary by Theorem \ref{Ihara formula}. 
\begin{corollary}\label{Ihara formula for a regular graph}
For $\left| u \right| < \frac{1}{\alpha}$, we have 
\begin{align*}
Z_{X}(u, x_{0}, x)&=(1-u^2)^{-\frac{q-1}{2}\delta_{x_{0}}(x)}\\
&\quad \times \exp\bigg(-\int_{\sigma(\Delta_{X})}\log(1-\big((q+1)-\lambda \big)u+qu^2)d\mu_{x_{0}, x}(\lambda)\bigg) \\
	&\quad \times \exp\bigg( \sum_{m=3}^{\infty}\frac{R_{m}(x_{0}, x)}{m}u^m\bigg). 
\end{align*}
\end{corollary}

Before we consider the case that $X$ is a finite $(q+1)$-regular graph, we introduce the notion of the local spectrum (\cite{Fiol}). 
For a vertex $x \in V\!X$, we denote $x$-local multiplicity of $\lambda_{i}$ by $\m_{x}(\lambda_{i})$. 
Here, the {\it$x$-local multiplicity} of $\lambda_{i}$ is the $xx$-entry of the primitive idempotent $E_{\lambda_{i}}$. Let $\{\mu_{0}=\lambda_{0}, \mu_{1}, \dots, \mu_{d_{x}}\}$ be the set of eigenvalues whose local multiplicities are positive. 
For each vertex $x \in V\!X$, we denote the $x$-local spectrum by $\sigma_{x}(X)$. 
Here, the {\it $x$-local spectrum} is $\sigma_{x}(X)=\{\lambda_{0}^{\m_{x}(\lambda_{0})}, \mu_{1}^{\m_{x}(\mu_{1})}, \dots, \mu_{d_{x}}^{\m_{x}(\mu_{d_{x}})}\}$. 
Then, we have the following corollary immediately by Corollary \ref{Ihara formula for a regular graph}. 
\begin{corollary}
For $\left| u \right| < \frac{1}{\alpha}$, we have 
\begin{align*}
Z_{X}(u, x_{0})
&=(1-u^2)^{-\frac{q-1}{2}}
	\prod_{\lambda \in \sigma_{x_{0}}(\Delta_{X})}
	\big(1-(q+1-\lambda)u+qu^2\big)^{-m_{x_{0}}(\lambda)}\\
	&\quad \times\exp\bigg(\sum_{m=3}^{\infty}\frac{R_{m}(x_{0})}{m}u^m\bigg). 
\end{align*}
\end{corollary}
We remark that this formula holds for a regular graph which is not always a simple graph by Proposition \ref{Nformula} and the formula obtained in p.~$188$ in \cite{CJK1}. 
This formula is regarded as a local version of the original Ihara determinant formula since the above equation gives the explicit relationship between the number of closed geodesics starting at $x_{0}$ and the $x_{0}$-local spectrum of the Laplacian of $X$. 

\section*{Acknowledgment}
The author expresses gratitude to Professor Hiroyuki Ochiai for his many helpful comments.


\begin{thebibliography}{20}
\bibitem{Bass1992} H.~Bass, The Ihara-Selberg zeta function of a tree lattice, International.~J.~Math. $3$ $(1992)$, $717$--$797$. 

\bibitem{CJK1} G.~Chinta, J.~Jorgenson and A.~Karlsson, Heat kernels on regular graphs and generalized Ihara zeta function formulas, Monatsh. Math., 178 $(2015)$, $171$--$190$. 

\bibitem{CMS2001} B.~Clair and S.~Mokhtari-Sharghi, Zeta functions of discrete groups acting on trees, J.~Algebra $237$ $(2001)$, No.~2, $561$--$620$. 

\bibitem{Cla2009} B.~Clair, Zeta functions of graphs with $\Z$ actions, J.~Combin.~Theory Ser.~B $99$ $(2009)$, No.~1, $48$--$61$. 

\bibitem{Deit2015} A.~Deitmar, Ihara zeta functions of infinite weighted graphs, SIAM J.~Discrete Math., $29(4)$ $(2015)$, $2100$--$2116$. 

\bibitem{Fiol} M.~A.~Fiol, E.~Garriga and J.~L.~A.~Yebra, Locally pseudo-distance-regular graphs, J.~Combin.~Theory Ser.~B $68$ $(1996)$ $179-205$. 

\bibitem{GZ2004} R.~I.~Grigorchuk and A.~\.Zuk, The Ihara zeta function of infinite graphs, the KNS spectral measure and integrable maps, Random walks and geometry, Walter de Gruyter GmbH \& Co.~KG, Berlin, $2004$, pp.~$141$--$180$. 

\bibitem{GIL2008a} D.~Guido, T.~Isola and M.~L.~Lapidus, Ihara zeta functions for periodic simple graphs, $C^{\star}$-algebras and elliptic theory $I\hspace{-.1em}I$, Trends Math., Birkh\"auser, Basel, $2008$, pp.~$103$--$121$. 

\bibitem{GIL2008b} D.~Guido, T.~Isola and M.~L.~Lapidus, Ihara's zeta function for periodic graphs and its approximation in amenable case, J.~Funct.~Anal.~$255$ $(2008)$, No.~$6$, $1339$--$1361$. 

\bibitem{HH89} K.~Hashimoto and A.~Hori, Selberg-Ihara's zeta function for $p$-adic groups, Automorphic forms and geometry of arithmetic varieties, Adv.~Stud.~Pure Math., vol.~$15$, Academic Press, Boston, MA, $1989$, pp.~$171$--$210$. 

\bibitem{H89} K.~Hashimoto, Zeta functions of finite graphs and representations of $p$-adic groups, Automorphic forms and geometry of arithmetic varieties, Adv.~Stu.~Pure Math., vol.~$15$, Academic Press, Boston, MA, $1989$, pp.~$211$--$280$. 

\bibitem{H90} K.~Hashimoto, On zeta and L-functions of finite graphs, Internet.~J.~Math.~$1$ $(1990)$, no.~$4$, $381$--$396$. 

\bibitem{H92} K.~Hashimoto, Artin type L-functions and the density theorem for prime cycles on finite graphs, Internet.~J.~Math.~$3$ $(1992)$, no.~$6$, $809$--$826$. 

\bibitem{H93} K.~Hashimoto, Artin L-functions of finite graphs and their applications, S\=urikaisekikenky\=usho K\=oky\=uroku $840$ $(1993)$, $70$--$81$. Algebraic combinatorics (Kyoto, $1992$). 

\bibitem{YI1966} Y.~Ihara, On discrete subgroups of the two by two projective linear group over $p$-adic field, J.~Math.~Soc.~Japan $18$ $(1966)$, $219$--$235$. 

\bibitem{KS2000} M.~Kotani and T.~Sunada, Zeta functions of finite graphs, J.~Math.~Sci.~Univ.~Tokyo $7$ $(2000)$, $7$--$25$. 

\bibitem{Sch1999} O.~Scheja, On zeta functions of arithmetically defined graphs, Finite Fields Appl.~$5$ $(1999)$, No.~$3$, $314$--$343$. 

\bibitem{serre1} J.~P.~Serre, Trees, Springer Monographs in Mathematics. Springer-Verlag, Berlin, $2003$. 

\bibitem{sunada86} T.~Sunada, L-Functions in geometry and some applications, Curvature and topology of Riemannian manifolds (Katata, $1985$), Lecture Notes in Math., vol.~$1201$, Springer, Berlin, $1986$, pp.~$266$--$284$. 

\bibitem{sunada88} T.~Sunada, Fundamental groups and Laplacians, Geometry and analysis on manifolds (Katata/Kyoto, $1987$), Lecture Notes in Math., vol.~$1339$, Springer, Berlin, $1988$, pp.~$248$--$277$. 

\bibitem{sunada1} T.~Sunada, Topological Crystallography, Surveys and Tutorials in the Applied Mathematical Sciences volume $6$, Springer-Berlin, $2013$. 

\bibitem{Terras} A.~Terras, Zeta functions of graphs. A stroll through the garden, Cambridge Studies in Advanced Mathematics, $128$. Cambridge University Press, Cambridge, $2011$. 


\end{thebibliography}
\end{document}